\documentclass[10pt]{article}

\usepackage{hyperref}
\usepackage{xcolor}
\usepackage{amssymb}
\usepackage{amsfonts}
\usepackage{amsthm}
\usepackage{amsmath}
\usepackage{float}
\usepackage{bigints}
\usepackage{fullpage}
\usepackage{mathtools}
\usepackage[utf8]{inputenc} 
\usepackage{cancel}
\usepackage{verbatim}
\usepackage{soul}
\usepackage{color}
\usepackage{tikz}
\usetikzlibrary{calc}
\usepackage{caption}
\usepackage{subcaption}
\usepackage{pgfplots}
 
\pgfplotsset{compat = newest}

\newtheorem{theorem}{Theorem}

\newtheorem{lemma}[theorem]{Lemma}
\newtheorem{sublemma}{Sublemma}

\newtheorem{conjecture}[theorem]{Conjecture}
\newtheorem{definition}{Definition}

\newtheorem{corollary}[theorem]{Corollary}
\newtheorem{remark}[theorem]{Remark}

\newtheorem{claim}[theorem]{Claim}

\newcommand{\bd}{\rm bd}

\newcommand{\x}{\mathbf{x}}
\newcommand{\p}{\mathbf{p}}
\newcommand{\q}{\mathbf{q}}
\newcommand{\uu}{\mathbf{u}}
\newcommand{\vv}{\mathbf{v}}
\newcommand{\y}{\mathbf{y}}
\newcommand{\z}{\mathbf{z}}
\newcommand{\oo}{\mathbf{o}}
\newcommand{\aaa}{\mathbf{a}}
\newcommand{\bbb}{\mathbf{b}}

\newcommand{\ee}{\mathbf{e}}

\newcommand{\B}{\mathbf{B}}

\newcommand{\K}{\mathbf{K}}

\newcommand{\Ee}{\mathbb{E}}

\newcommand{\Ss}{\mathbb{S}}

\newcommand{\Rr}{\mathbb{R}}

\newcommand{\Ed}{\Ee^d}
\newcommand{\Rd}{\Rr^d}
\newcommand{\Sd}{\Ss^{d-1}}

\newcommand{\noshow}[1]{}
\newcommand{\Sedm}{{\mathbb S}^{d-1}}

\newcommand{\capp}[2]{C_{{\mathbb S}^{d-1}}\left[#1, #2\right]}
\newcommand{\copp}[2]{C_{{\mathbb S}^{d-1}}\left(#1, #2\right)}

\title{Illuminating spiky balls and cap bodies
\footnote{Keywords and phrases: Euclidean $d$-space, spiky ball, $2$-illuminable, cap body, unconditionally symmetric, illumination number, illumination conjecture. 
\newline \hspace*{.35cm} 2010 Mathematics Subject Classification: 52A20, 52A22.}
\author{K\'{a}roly Bezdek\thanks{Partially supported by a Natural Sciences and 
Engineering Research Council of Canada Discovery Grant.}, Ilya Ivanov, and Cameron Strachan}}

\date{}

\begin{document}

\maketitle

\begin{abstract}
\noindent The convex hull of a ball with an exterior point is called a spike (or cap). A union of finitely many spikes of a ball is called a spiky ball. If a spiky ball is convex, then we call it a cap body. In this note we upper bound the illumination numbers of $2$-illuminable spiky balls as well as centrally symmetric cap bodies. In particular, we prove the Illumination Conjecture for centrally symmetric cap bodies in sufficiently large dimensions by showing that any $d$-dimensional centrally symmetric cap body can be illuminated by $<2^d$ directions in Euclidean $d$-space for all $d\geq 20$. Furthermore, we strengthen the latter result for $1$-unconditionally symmetric cap bodies.
\end{abstract}

\section{Introduction}\label{sec:intro}

Let $\Ee^d$ denote the $d$-dimensional Euclidean vector space, with inner product $\langle\cdot ,\cdot\rangle$ and norm $\|\cdot\|$ and let $\ee_1, \dots, \ee_d$ be its standard basis. Its unit sphere centered at the origin $\oo$ is $\Sedm:= \{\x\in\Ee^d\ |\ \|\x\|= 1\}$. A {\it greatcircle} of $\Sedm$ is an intersection of $\Sedm$ with a plane of $\Ee^d$ passing through $\oo$.  Two points are called {\it antipodes} if they can be obtained as an intersection of $\Sedm$ with a line through $\oo$ in $\Ee^d$. If $\aaa ,\bbb\in\Sedm$ are two points that are not antipodes, then we label the (uniquely determined) shortest geodesic arc of $\Sedm$ connecting $\aaa$ and $\bbb$ by $\widehat{\aaa\bbb}$. In other words, $\widehat{\aaa\bbb}$ is the shorter circular arc with endpoints $\aaa$ and $\bbb$ of the greatcircle $\aaa\bbb$ that passes through $\aaa$ and $\bbb$. The length of $\widehat{\aaa\bbb}$ is called the {\it spherical} (or {\it angular}) {\it distance} between $\aaa$ and $\bbb$ and it is labelled by $l(\widehat{\aaa\bbb})$, where $0<l(\widehat{\aaa\bbb})<\pi$. The set $C_{{\mathbb S}^{d-1}}[\x, \alpha]:=\{\y\in{\mathbb S}^{d-1}\ |\ l(\widehat{\x,\y})\leq \alpha\}=\{\y\in{\mathbb S}^{d-1} | \langle\x ,\y\rangle\geq \cos\alpha\}$ (resp., $C_{{\mathbb S}^{d-1}}(\x, \alpha):=\{\y\in{\mathbb S}^{d-1}\ |\ l(\widehat{\x,\y})< \alpha\}=\{\y\in{\mathbb S}^{d-1} | \langle\x ,\y\rangle> \cos\alpha\}$) is called the closed (resp., open) {\it spherical cap} of angular radius $\alpha$ centered at $\x\in{\mathbb S}^{d-1}$ for $0<\alpha\leq \frac{\pi}{2}$.
The closed Euclidean ball of radius $r$ centered at $\p\in\Ed$ is denoted by $\B^d[\p,r]:=\{\q\in\Ed\ |\  |\p-\q|\leq r\}$. A $d$-dimensional \textit{convex body} $\mathbf{K}$ is a compact convex subset of ${\mathbb{E}}^{d}$ with non-empty interior. Then $\mathbf{K}$ is said to be \textit{$\mathbf{o}$-symmetric} if $\mathbf{K}=-\mathbf{K}$ and $\mathbf{K}$ is called \textit{centrally symmetric} if some translate of $\mathbf{K}$ is $\mathbf{o}$-symmetric. A light source at a point $\mathbf{p}$ outside a convex body $\mathbf{K}\subset{\mathbb{E}}^{d}$, \textit{illuminates} a point $\mathbf{x}$ on the boundary of $\mathbf{K}$ if the halfline originating from $\mathbf{p}$ and passing through $\mathbf{x}$ intersects the interior of $\mathbf{K}$ at a point not lying between $\mathbf{p}$ and $\mathbf{x}$. The set of points $\{\mathbf{p}_{i}:i=1,\ldots,n\}$ in the exterior of $\mathbf{K}$ is said to \textit{illuminate} $\mathbf{K}$ if every boundary point of $\mathbf{K}$ is illuminated by some $\mathbf{p}_{i}$. The \textit{illumination number} $I(\mathbf{K})$ of $\mathbf{K}$ is the smallest $n$ for which $\mathbf{K}$ can be illuminated by $n$ point light sources. One can also consider illumination of $\mathbf{K}\subset{\mathbb{E}}^{d}$ by directions instead of by exterior points. We say that a point $\mathbf{x}$ on the boundary of $\mathbf{K}$ is illuminated in the direction $\mathbf{v}\in\mathbb{S}^{d-1} $ if the halfline originating from $\mathbf{x}$ and with direction vector $\mathbf{v}$ intersects the interior of $\mathbf{K}$. The former notion of illumination was introduced by Hadwiger \cite{hadwiger2}, while the latter notion is due to Boltyanski \cite{boltyanski1}. It may not come as a surprise that the two concepts are equivalent in the sense that a convex body $\mathbf{K}$ can be illuminated by $n$ point sources if and only if it can be illuminated by $n$ directions. The following conjecture of Boltyanski \cite{boltyanski1} and Hadwiger \cite{hadwiger2} has become a central problem of convex and discrete geometry and inspired a significant body of research. 
\begin{conjecture}[\textbf{Illumination Conjecture}]\label{illumination conjecture}
The illumination number $I(\mathbf{K})$ of any $d$-dimensional convex body $\mathbf{K}$, $d\geq 3$, is at most $2^d$ and $I(\mathbf{K}) = 2^d$ only if $\mathbf{K}$ is an affine $d$-cube.
\end{conjecture}

While Conjecture~\ref{illumination conjecture} has been proved in the plane (\cite{boltyanski1}, \cite{gohberg1}, \cite{hadwiger2}, and \cite{lev1}), it is open for dimensions larger than $2$. On the other hand, there are numerous partial results supporting Conjecture~\ref{illumination conjecture} in dimensions greater than $2$. For details we refer the interested reader to the recent survey article \cite{BeKh} and the references mentioned there. Here we highlight only the following results. Let $\mathbf{K}$ be an arbitrary $d$-dimensional convex body with $d>1$. Rogers \cite{Ro} (see also \cite{RoZo}) has proved that $I(\mathbf{K})\leq \binom{2d}{d}d(\ln d+\ln\ln d +5)=O(4^d\sqrt{d}\ln d)$. Huang, Slomka, Tkocz, and Vritsiou \cite{HSTV} improved this bound of Rogers for sufficiently large values of $d$ to $c_14^de^{-c_2\sqrt{d}}$,
where $c_1,c_2>0$ are universal constants. Lassak \cite{Las} improved the upper bound of Rogers for some small values of $d$ to $ (d+1)d^{d-1}-(d-1)(d-2)^{d-1}$. In fact, the best upper bounds for the illumination numbers of convex bodies in dimensions $3,4,5,6$ are $16$ (\cite{Pa}), $96, 1091, 15373$ (\cite{PS}). The best upper bound for the illumination numbers of {\it centrally symmetric} convex bodies of $\Ee^d$, $d>1$ is $2^dd(\ln d+\ln\ln d +5)$ proved by Rogers (\cite{Ro} and \cite{RoZo}). In connection with this upper bound we note that \cite{Ti} proves Conjecture~\ref{illumination conjecture} for \textit{unit balls of $1$-symmetric norms} in $\Rd$ provided that $d$ is sufficiently large. We also mention in passing that Conjecture~\ref{illumination conjecture} has been confirmed for certain classes of convex bodies such as {\it wide ball-bodies} including {\it convex bodies of constant width} (\cite{BeKi}, \cite{Be11}, \cite{Be12}, \cite{BPR}, \cite{Sch}), {\it convex bodies of Helly dimension $2$} (\cite{bo3}), and {\it belt-bodies} including {\it zonoids} and {\it zonotopes} (\cite{bol}). The present article has been motivated by the investigations in \cite{Na} and it aims at proving Conjecture~\ref{illumination conjecture} for sufficiently high dimensional {\it centrally symmetric cap bodies} studied under the name centrally symmetric spiky balls in \cite{Na}. Actually, we do a bit more. The details are as follows. 

\begin{definition}\label{spiky ball and cap body}
Let $\B^d:=\B^d[\oo,1]$ and let $\x_1,\dots ,\x_n\in\Ee^d\setminus\B^d$. Then $${\rm Sp}_{\B^d}[\x_1,\dots ,\x_n]:=\bigcup_{i=1}^n{\rm conv}(\B^d\cup\{\x_i\})$$ is called a {\rm spiky (unit) ball}, where ${\rm conv}(\cdot)$ refers to the convex hull of the corresponding set. If $\x_i\notin \bigcup_{1\leq j\leq n, j\neq i}{\rm conv}(\B^d\cup\{\x_j\})$ holds for some $1\leq i\leq n$, then  $\x_i$ is called a {\rm vertex} of ${\rm Sp}_{\B^d}[\x_1,\dots ,\x_n]$.  A point $\mathbf{x}$ on the boundary of the spiky ball ${\rm Sp}_{\B^d}[\x_1,\dots ,\x_n]$ is {\rm illuminated in the direction} $\mathbf{v}\in\mathbb{S}^{d-1} $ if the halfline originating from $\mathbf{x}$ and with direction vector $\mathbf{v}$ intersects the interior of ${\rm Sp}_{\B^d}[\x_1,\dots ,\x_n]$ in points arbitrarily close to $\x$. Furthermore, the set of directions $\{\mathbf{v}_{i}:i=1,\ldots,m\}\subset {\mathbb S}^{d-1}$ is said to {\rm illuminate} ${\rm Sp}_{\B^d}[\x_1,\dots ,\x_n]$ if every boundary point of ${\rm Sp}_{\B^d}[\x_1,\dots ,\x_n]$ is illuminated by some $\mathbf{v}_{i}$. The {\rm illumination number} $I({\rm Sp}_{\B^d}[\x_1,\dots ,\x_n])$ of ${\rm Sp}_{\B^d}[\x_1,\dots ,\x_n]$ is the smallest $m$ for which ${\rm Sp}_{\B^d}[\x_1,\dots ,\x_n]$ can be illuminated by $m$ directions. Moreover, we say that the spiky ball ${\rm Sp}_{\B^d}[\x_1,\dots ,\x_n]$ with vertices $\x_1,\dots ,\x_n$ is {\rm $2$-illuminable} if any two of its vertices can be simultaneously illuminated by a direction in $\Ee^d$. Finally, ${\rm Sp}_{\B^d}[\x_1,\dots ,\x_n]$ is called a {\rm cap body} if it is a convex body in $\Ee^d$.(See Figure~\ref{fig:test}.)
\end{definition}

 We note that cap bodies were first studied by Minkowski \cite{mink}.

\begin{figure}
\centering
\begin{subfigure}{.5\textwidth}
  \centering
      \begin{tikzpicture} 
\node [above right, inner sep=0] (image) at (0,0) {\includegraphics[width=0.8\linewidth]{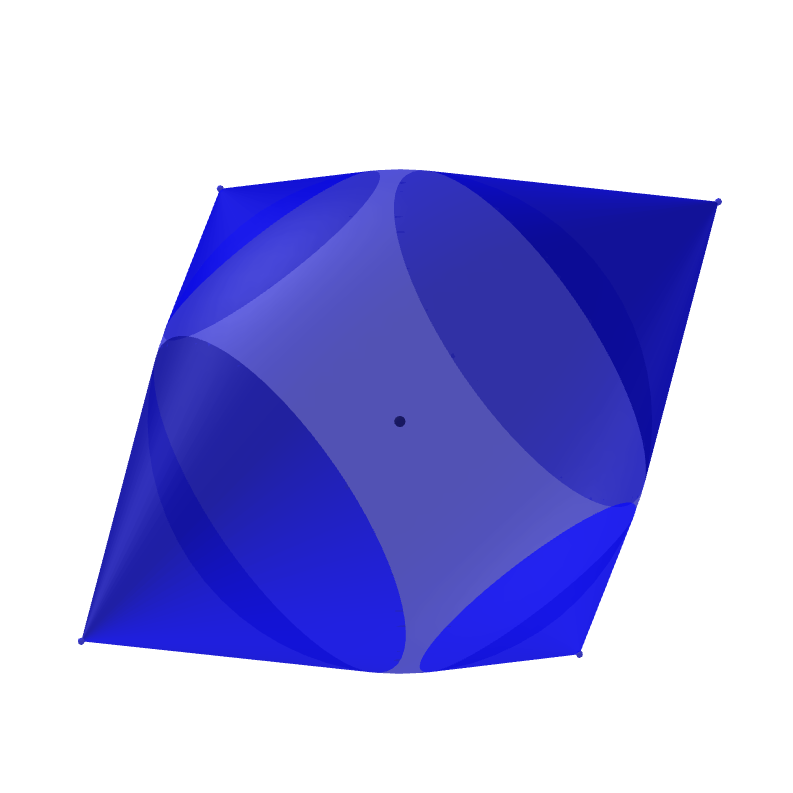}};    
   
\begin{scope}[
x={($0.1*(image.south east)$)},
y={($0.1*(image.north west)$)}];
    
    \node[left] at (3,8){\large $\x_1$};    
    \node[below left] at (1,2){\large $\x_2$};    
    \node[below right] at (7,1.75){\large $-\x_1$};
    \node[above right] at (9,7){\large $-\x_2$};
    \node[below right] at (5,5.25){\large $\oo$};
    
\end{scope}
\end{tikzpicture}
  \caption{An $\oo$-symmetric convex spiky ball, i.e., a cap body.}
  
\end{subfigure}%
\begin{subfigure}{.5\textwidth}
  \centering
      \begin{tikzpicture} 
\node [above right, inner sep=0] (image) at (0,0) {\includegraphics[width=0.8\linewidth]{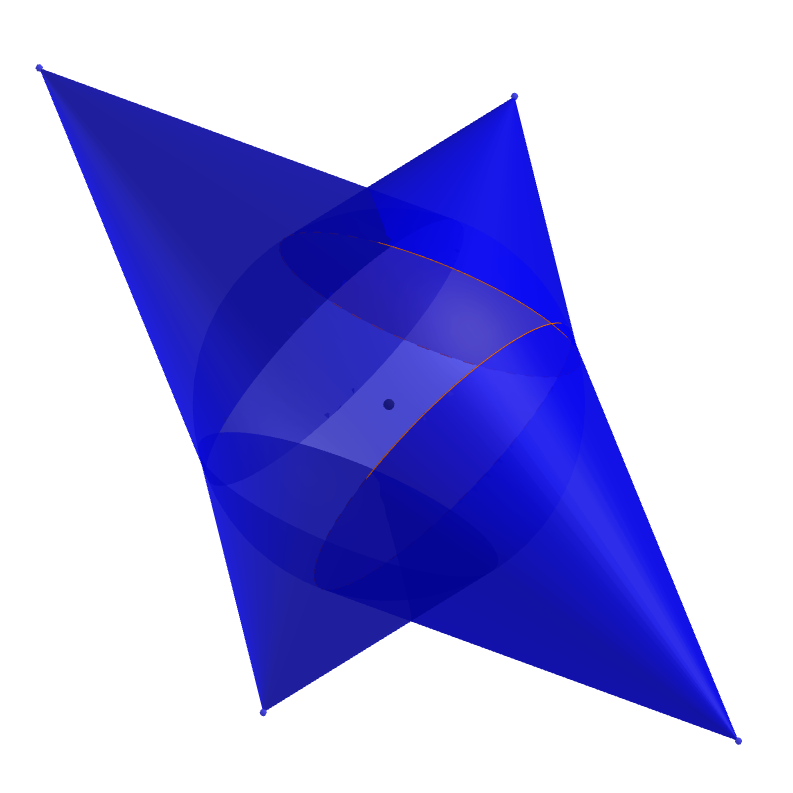}};    
    
\begin{scope}[
x={($0.1*(image.south east)$)},
y={($0.1*(image.north west)$)}];
        
    \node[left] at (1,9.5){\large $\x_1$};    
    \node[below left] at (4,1){\large $\x_2$};    
    \node[below right] at (9.25,1){\large $-\x_1$};
    \node[above right] at (6.5,8.5){\large $-\x_2$};
    \node[above right] at (4.75,5){\large $\oo$};

\end{scope}
\end{tikzpicture}
  \caption{An $\oo$-symmetric non-convex spiky ball.}
  
\end{subfigure}
\caption{Centrally symmetric spiky balls.}
\label{fig:test}
\end{figure}
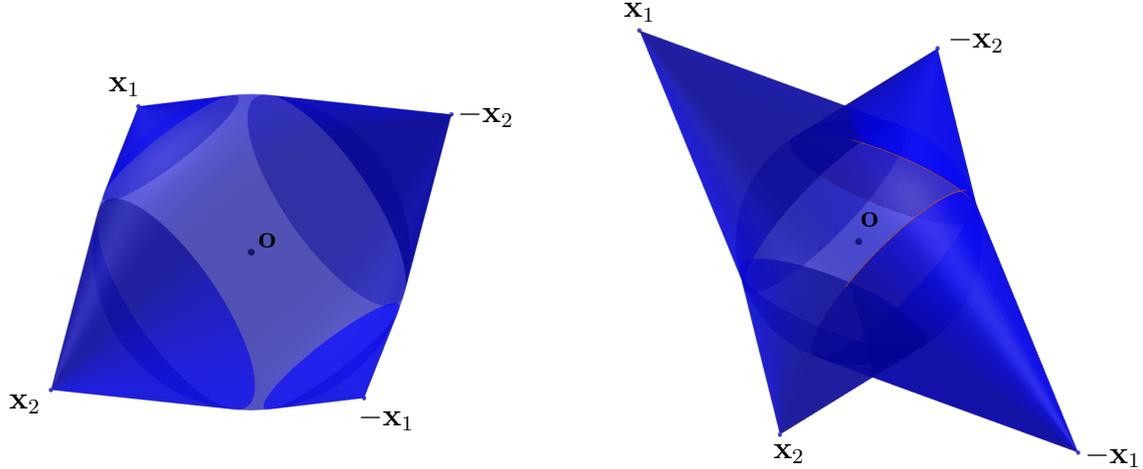

\begin{definition}\label{covering number}
If $0<\alpha\leq\frac{\pi}{2}$, then let $N_{\Sedm}(\alpha)$ denote the minimum number of closed spherical caps of angular radius $\alpha$ that can cover $\Sedm$.\end{definition}

Our first result upper bounds the illumination numbers of $2$-illuminable spiky balls. We note that spiky balls without being $2$-illuminable can have arbitrarily large illumination numbers.

\begin{theorem}\label{illumination-spiky balls}
Suppose that ${\rm Sp}_{\B^d}[\x_1,\dots ,\x_n]$ is a $2$-illuminable spiky ball with vertices $\x_1,\dots ,\x_n$ in $\Ee^d$.
\begin{itemize}
\item[(i)] If $d=2$, then $I({\rm Sp}_{\B^2}[\x_1,\dots ,\x_n])=3$.
\item[(ii)] If $d=3$, then $I({\rm Sp}_{\B^3}[\x_1,\dots ,\x_n])\leq 5$.
\item[(iii)] If $d\geq 4$, then $I({\rm Sp}_{\B^d}[\x_1,\dots ,\x_n])\leq 3+N_{{\mathbb S}^{d-2}}(\frac{\pi}{6})$.
\end{itemize}
\end{theorem}
\begin{corollary}\label{estimates-spiky balls}
Let ${\rm Sp}_{\B^d}[\x_1,\dots ,\x_n]$ be a $2$-illuminable spiky ball with vertices $\x_1,\dots ,\x_n$ in $\Ee^d, d\geq 4$. If $d=4$, then $I({\rm Sp}_{\B^4}[\x_1,\dots ,\x_n])\leq 23$. If $d\geq 5$, then $$I({\rm Sp}_{\B^d}[\x_1,\dots ,\x_n])\leq 3+2^{d-2}\sqrt{2\pi(d-1)}\left(\frac{1}{2}+\frac{2\ln\ln(d-2)}{\ln(d-2)}+\frac{5}{\ln(d-2)}\right)(d-2)\ln(d-2)<2^{d+1}d^{\frac{3}{2}}\ln d.$$ 
\end{corollary}
\begin{remark}
Note that an arbitrary spiky ball ${\rm Sp}_{\B^d}[\x_1,\dots ,\x_n]$ is starshaped with respect to $\oo$ and even if ${\rm Sp}_{\B^d}[\x_1,\dots ,\x_n]$ is $2$-illuminable it is not necessarily a convex set. Still, one may wonder whether any $d$-dimensional $2$-illuminable spiky ball can be illuminated by less than $2^d$ directions in $\Ee^d$.
\end{remark}

\begin{remark}\label{construction1}
There exists a $2$-illuminable spiky ball ${\rm Sp}_{\B^3}[\x_1,\dots ,\x_{10}]$ in $\Ee^3$ with $I({\rm Sp}_{\B^3}[\x_1,\dots ,\x_{10}])=5$. Furthermore, there exists $d_0$ such that for any $d\geq d_0$ one possesses a $2$-illuminable spiky ball ${\rm Sp}_{\B^d}[\x_1,\dots ,\x_n]$ in $\Ee^d$ with $I({\rm Sp}_{\B^d}[\x_1,\dots ,\x_{n}])>1.0645^{d-1}$. 
\end{remark}

Before we state our main result let us recall the following very interesting theorem of Nasz\'odi \cite{Na}: Let $1<D< 1.116$. Then for any sufficiently large dimension $d$ there exists a centrally symmetric cap body $\K$ such that $I(\K)\geq 0.05D^d$ and $\frac{1}{D}\B^d\subset\K\subset \B^d$. This raises the natural question whether Conjecture~\ref{illumination conjecture} holds for centrally symmetric cap bodies in sufficiently large dimensions. We give a positive answer this question as follows.

\begin{theorem}\label{illumination-cap bodies}
Let ${\rm Sp}_{\B^d}[\pm\x_1,\dots ,\pm\x_n]$ be an $\oo$-symmetric cap body with vertices $\pm\x_1,\dots ,\pm\x_n$ in $\Ee^d$, $d\geq 3$. Then 
$$I({\rm Sp}_{\B^d}[\pm\x_1,\dots ,\pm\x_n])\leq 2+N_{{\mathbb S}^{d-2}}\left(\frac{\pi}{4}\right).$$
\end{theorem}

\begin{corollary}\label{estimates-cap bodies}
Any $3$-dimensional centrally symmetric cap body can be illuminated by $6\ (<2^3)$ directions in $\Ee^3$. (This is not a new result. It was proved via a dual method in \cite{IvSt}.) On the other hand, any $4$-dimensional centrally symmetric cap body can be illuminated by $12\ (<2^4)$ directions in $\Ee^4$. Moreover, if ${\rm Sp}_{\B^d}[\pm\x_1,\dots ,\pm\x_n]$ is an $\oo$-symmetric cap body with vertices $\pm\x_1,\dots ,\pm\x_n$ in $\Ee^d, d\geq 20$, then
$$I({\rm Sp}_{\B^d}[\pm\x_1,\dots ,\pm\x_n])\leq 2+2^{\frac{d-2}{2}}\sqrt{2\pi(d-1)}\left(\frac{1}{2}+\frac{2\ln\ln(d-2)}{\ln(d-2)}+\frac{5}{\ln(d-2)}\right)(d-2)\ln(d-2)<2^d.$$ \end{corollary}

\begin{definition}
The cap body $\K\subset\Ee^d$ is called {\rm $1$-unconditionally symmetric} if it symmetric about each coordinate hyperplane of $\Ee^d$. 
\end{definition}

We close this section with a strengthening of Corollary \ref{estimates-cap bodies} for $1$-unconditionally symmetric cap bodies. Recall that according to \cite{IvSt} if $\K$ is a $1$-unconditionally symmetric cap body in $\Ee^4$, then $I(\K)\leq 8$. 

\begin{theorem}\label{unconditionally symmetric}
Let $\K$ be a $1$-unconditionally symmetric cap body in $\Ee^d, d\geq 5$. Then $I(\K)\leq 4d$.
\end{theorem} 

While this proves Conjecture~\ref{illumination conjecture} for $1$-unconditionally symmetric cap bodies in dimensions $d\geq 5$, the $4d$ estimate does not seem to be sharp, and, in fact, we propose 

\begin{conjecture}\label{un-symm-2d}
Every $1$-unconditionally symmetric cap body of $\Ed$ can be illuminated by $2d$ directions for all $d\geq 5$. 
\end{conjecture}

In the rest of the paper we prove the theorems stated.

\section{Proof of Theorem~\ref{illumination-spiky balls}}

We start with

\begin{definition}\label{underlying-open-caps}
If ${\rm Sp}_{\B^d}[\x_1,\dots ,\x_n]$ is a spiky ball with vertices $\x_1,\dots ,\x_n$ in $\Ee^d$, then let $C_{{\mathbb S}^{d-1}}(\y_i, \alpha_i):={\rm int}\left({\rm conv}(\B^d\cup\{\x_i\})\right)\cap {\mathbb S}^{d-1}$ denote the open spherical cap assigned to the vertex $\x_i$ of ${\rm Sp}_{\B^d}[\x_1,\dots ,\x_n]$, where $1\leq i\leq n$, $0<\alpha_i<\frac{\pi}{2}$ and ${\rm int}(\cdot)$ refers to the interior of the corresponding set in $\Ee^d$.   
\end{definition}

It is easy to see that the direction $\mathbf{v} \in {\mathbb S}^{d-1}$ illuminates the vertex $\x_i$ of the spiky ball ${\rm Sp}_{\B^d}[\x_1,\dots ,\x_n]$ if and only if $\mathbf{v}\in C_{{\mathbb S}^{d-1}}(-\y_i, \frac{\pi}{2}-\alpha_i)$. Thus, by observing that the set $\{\mathbf{v}_{k}: 1\leq k\leq m\}\subset {\mathbb S}^{d-1}$ of directions whose positive hull ${\rm pos}(\{\mathbf{v}_{k}: 1\leq k\leq m\}):=\{\sum_{k=1}^m\lambda_k\mathbf{v}_k\ |\ \lambda_k>0\ {\rm for\ all}\ 1\leq k\leq m\}$ is $\Ee^d$, illuminates the spiky ball ${\rm Sp}_{\B^d}[\x_1,\dots ,\x_n]$ if and only if it illuminates the vertices $\x_1, \dots ,\x_n$ of  ${\rm Sp}_{\B^d}[\x_1,\dots ,\x_n]$, the following statement is immediate.

\begin{lemma}\label{intersecting}
Let ${\rm Sp}_{\B^d}[\x_1,\dots ,\x_n]$ be a spiky (unit) ball with vertices $\x_1,\dots ,\x_n$ in $\Ee^d$. Then
\begin{itemize}
\item[(a)] ${\rm Sp}_{\B^d}[\x_1,\dots ,\x_n]$ is  $2$-illuminable if and only if $C_{{\mathbb S}^{d-1}}(-\y_i, \frac{\pi}{2}-\alpha_i)\cap C_{{\mathbb S}^{d-1}}(-\y_j, \frac{\pi}{2}-\alpha_j)\neq\emptyset$ holds for all $1\leq i<j\leq n$ moreover,
\item[(b)] $\{\mathbf{v}_{k}: 1\leq k\leq m\}\subset {\mathbb S}^{d-1}$ with ${\rm pos}(\{\mathbf{v}_{k}: 1\leq k\leq m\})=\Ee^d$ illuminates ${\rm Sp}_{\B^d}[\x_1,\dots ,\x_n]$ if and only if $C_{{\mathbb S}^{d-1}}(-\y_i, \frac{\pi}{2}-\alpha_i)\cap\{\mathbf{v}_{k}: 1\leq k\leq m\}\neq\emptyset$ holds for all $1\leq i\leq n$.
\end{itemize}
\end{lemma}
Now, we are set to prove Theorem~\ref{illumination-spiky balls}.

\noindent {\it Part (i)}:  Let ${\rm Sp}_{\B^2}[\x_1,\dots ,\x_n]$ be a $2$-illuminable spiky (unit) disk with vertices $\x_1,\dots ,\x_n$ in $\Ee^2$. Let $\mathcal{C}:=\{ C_{{\mathbb S}^{1}}(-\y_i, \frac{\pi}{2}-\alpha_i)\ |\ 1\leq i\leq n\}$ be the family of open circular arcs (of length $<\pi$) assigned to the vertices of ${\rm Sp}_{\B^2}[\x_1,\dots ,\x_n]$. Without loss of generality we may assume that  $C_{{\mathbb S}^{1}}(-\y_1, \frac{\pi}{2}-\alpha_1)$ contains no other open circular arc of $\mathcal{C}$. As by Part (a) of Lemma~\ref{intersecting} $C_{{\mathbb S}^{1}}(-\y_i, \frac{\pi}{2}-\alpha_i)\cap C_{{\mathbb S}^{1}}(-\y_j, \frac{\pi}{2}-\alpha_j)\neq\emptyset$ holds for all $1\leq i<j\leq n$ therefore, there exist $\mathbf{v}_1, \mathbf{v}_2\in C_{{\mathbb S}^{1}}(-\y_1, \frac{\pi}{2}-\alpha_1)$ with each of them lying sufficiently close to one of the two endpoints of $C_{{\mathbb S}^{1}}(-\y_1, \frac{\pi}{2}-\alpha_1)$ such that $C_{{\mathbb S}^{1}}(-\y_i, \frac{\pi}{2}-\alpha_i)\cap\{\mathbf{v}_1, \mathbf{v}_2\}\neq\emptyset$ holds for all $1\leq i\leq n$. Clearly, $\mathbf{v}_1\neq -\mathbf{v}_2$ and so, one can choose $\mathbf{v}_3\in{\mathbb S}^{1}$ such that ${\rm pos}(\{\mathbf{v}_{k}: 1\leq k\leq 3\})=\Ee^2$. Hence, by  Part (b) of Lemma~\ref{intersecting} $\{\mathbf{v}_{k}: 1\leq k\leq 3\}\subset {\mathbb S}^{1}$ illuminates ${\rm Sp}_{\B^2}[\x_1,\dots ,\x_n]$, implying $I({\rm Sp}_{\B^2}[\x_1,\dots ,\x_n])=3$ in a straightforward way.

\noindent {\it Part (ii)}: Let ${\rm Sp}_{\B^3}[\x_1,\dots ,\x_n]$ be a $2$-illuminable spiky (unit) ball with vertices $\x_1,\dots ,\x_n$ in $\Ee^3$. Let $\mathcal{C}:=\{ C_{{\mathbb S}^{2}}(-\y_i, \frac{\pi}{2}-\alpha_i)\ |\ 1\leq i\leq n\}$ be the family of open spherical caps assigned to the vertices of ${\rm Sp}_{\B^3}[\x_1,\dots ,\x_n]$. By Part (a) of Lemma~\ref{intersecting} any two members of $\mathcal{C}$ intersect. Next, recall the following theorem of Danzer \cite{Da}: If $\mathcal{F}$ is a family of finitely many closed spherical caps on ${\mathbb S}^{2}$ such that every two members of $\mathcal{F}$ intersect, then there exist $4$ points on ${\mathbb S}^{2}$ such that each member of $\mathcal{F}$ contains at least one of them (i.e., $4$ needles are always sufficient to pierce all members of $\mathcal{F}$). Now, applying Danzer's theorem to $\mathcal{C}$ (or rather to the corresponding family of closed spherical caps with each closed spherical cap being somewhat smaller and concentric to an open spherical cap of $\mathcal{C}$) one obtains the existence of $\mathbf{v}_1, \mathbf{v}_2, \mathbf{v}_3, \mathbf{v}_4  \in{\mathbb S}^{2}$ with the property that $\mathbf{v}_1, \mathbf{v}_2, \mathbf{v}_3$ are linearly independent and  $C_{{\mathbb S}^{2}}(-\y_i, \frac{\pi}{2}-\alpha_i)\cap\{\mathbf{v}_1, \mathbf{v}_2,  \mathbf{v}_3, \mathbf{v}_4  \}\neq\emptyset$ holds for all $1\leq i\leq n$. Finally, let us choose $\mathbf{v}_5 \in{\mathbb S}^{2}$ such that ${\rm pos}(\{\mathbf{v}_{k}: 1\leq k\leq 5\})=\Ee^3$. (See Figure~\ref{fig1}.) Thus,  Part (b) of Lemma~\ref{intersecting} implies in a straightforward way that $\{\mathbf{v}_{k}: 1\leq k\leq 5\}\subset {\mathbb S}^{2}$  illuminates ${\rm Sp}_{\B^3}[\x_1,\dots ,\x_n]$ and therefore $I({\rm Sp}_{\B^3}[\x_1,\dots ,\x_n])\leq 5$.
\begin{figure}
    \centering
\begin{tikzpicture} 
\node [above right, inner sep=0] (image) at (0,0) {\includegraphics[scale=0.3]{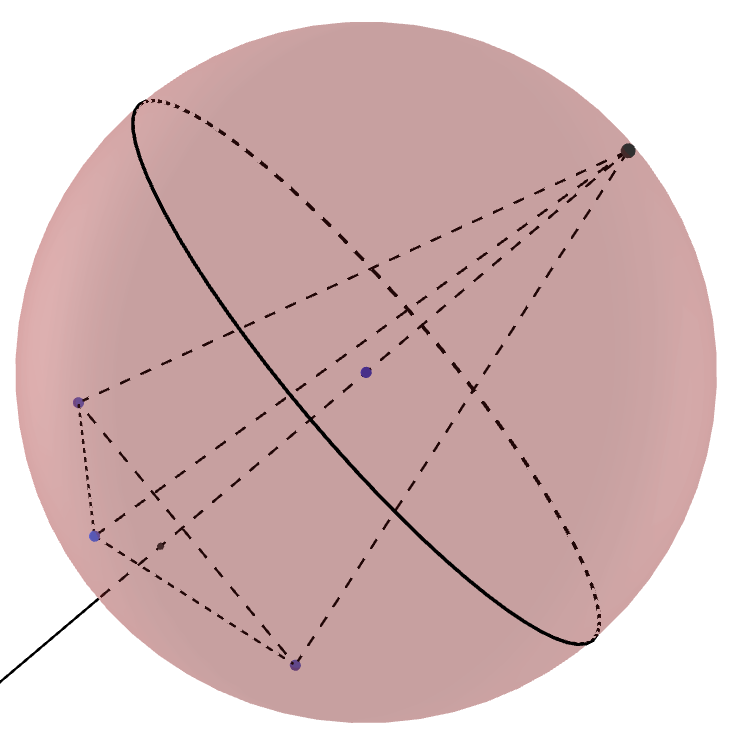}};    
    
\begin{scope}[
x={($0.1*(image.south east)$)},
y={($0.1*(image.north west)$)}];

    \node[above right] at (8.5,8){\large $\mathbf{v}_5$};
    \node[left] at (1.1,4.6){\large $\mathbf{v}_1$};    
    \node[below left] at (1.2,2.8){\large $\mathbf{v}_2$};    
    \node[below right] at (4,1){\large $\mathbf{v}_3$};
    \node[below right] at (5,5){\large $\oo$};
    
\end{scope}
\end{tikzpicture}

    \caption{Constructing $\mathbf{v}_5 \in{\mathbb S}^{2}$ in the proof of Part (ii) of Theorem~\ref{illumination-spiky balls}. }
    \label{fig1}
    
\end{figure}
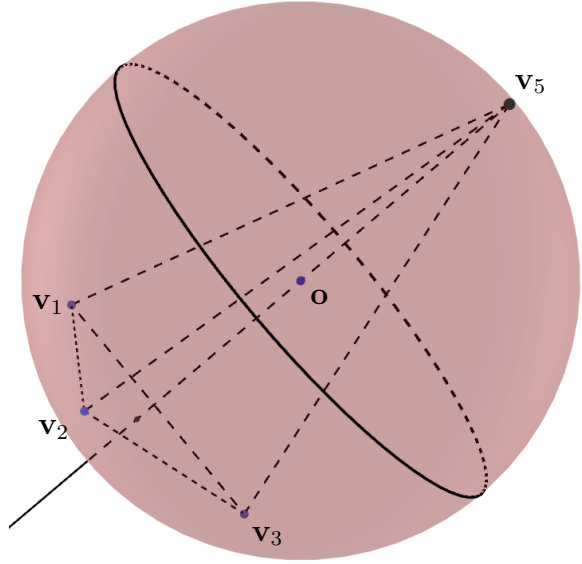

\noindent {\it Part (iii)}: Let ${\rm Sp}_{\B^d}[\x_1,\dots ,\x_n]$ be a $2$-illuminable spiky (unit) ball with vertices $\x_1,\dots ,\x_n$ in $\Ee^d$, $d\geq 4$. Let $\mathcal{C}:=\{ C_{{\mathbb S}^{d-1}}(-\y_i, \frac{\pi}{2}-\alpha_i)\ |\ 1\leq i\leq n\}$ be the family of open spherical caps assigned to the vertices of ${\rm Sp}_{\B^d}[\x_1,\dots ,\x_n]$. By Part (a) of Lemma~\ref{intersecting} any two members of $\mathcal{C}$ intersect. We need
\begin{definition}\label{Gallai number}
Let $G(2, \B^d)$ denote the smallest positive integer $k$ such that any finite family of pairwise intersecting $d$-dimensional closed balls in $\Ee^d$ is {\rm $k$-pierceable} (i.e., the finite family of balls can be partitioned into $k$ subfamilies each having a non-empty intersection).
\end{definition}
\noindent Now, recall Danzer's estimate (see page 361 in \cite{Ec}) according to which $G(2, \B^d)\leq 1+N_{{\mathbb S}^{d-1}}(\frac{\pi}{6})$. Let $\mathbf{s}\in {\mathbb S}^{d-1}$ be a point which is not a boundary point of any member of $\mathcal{C}$. If $\mathcal{C}'$ (resp., $\mathcal{C}''$) consists of those members of $\mathcal{C}$ that contain $\mathbf{s}$ as an interior (resp., exterior) point, then
clearly $\mathcal{C}=\mathcal{C}'\cup\mathcal{C}''$. Let $H$ be the hyperplane tangent to ${\mathbb S}^{d-1}$ at $-\mathbf{s}$ in $\Ee^d$. If we take the stereographic projection with center $\mathbf{s}$ that maps ${\mathbb S}^{d-1}\setminus\mathbf{s}$ onto $H$, then applying Danzer's estimate to the images of $\mathcal{C}''$ in $H$ we get that there are $1+N_{{\mathbb S}^{d-2}}(\frac{\pi}{6})$ points of ${\mathbb S}^{d-1}$ piercing the members of $\mathcal{C}''$ in ${\mathbb S}^{d-1}$. Hence, $\mathcal{C}$ is pierceable by $2+N_{{\mathbb S}^{d-2}}(\frac{\pi}{6})$ points (including $\mathbf{s}$) in ${\mathbb S}^{d-1}$. As members of $\mathcal{C}$ are open spherical caps of ${\mathbb S}^{d-1}$ therefore there are $3+N_{{\mathbb S}^{d-2}}(\frac{\pi}{6})$ points in ${\mathbb S}^{d-1}$ whose positive hull is $\Ee^d$ such that they pierce the members of $\mathcal{C}$. Thus, by Part (b) of Lemma~\ref{intersecting} we get that $I({\rm Sp}_{\B^d}[\x_1,\dots ,\x_n])\leq 3+N_{{\mathbb S}^{d-2}}(\frac{\pi}{6})$. This completes the proof of Theorem~\ref{illumination-spiky balls}.

\section{Proof of Corollary \ref{estimates-spiky balls}}

First, we recall that according to \cite{TaGa} there exists a covering of ${\mathbb S}^{2}$ using $20$ (closed) spherical caps of angular radius $\frac{\pi}{6}$. Thus, by Part (iii) of Theorem~\ref{illumination-spiky balls} if  ${\rm Sp}_{\B^4}[\x_1,\dots ,\x_n]$ is a $2$-illuminable spiky ball with vertices $\x_1,\dots ,\x_n$ in $\Ee^4$, then $I({\rm Sp}_{\B^4}[\x_1,\dots ,\x_n])\leq 23$.

Second, recall that Theorem 1 of \cite{Du} implies in a straightforward way that 
\begin{equation}\label{inequality1}
3+N_{{\mathbb S}^{d-2}}\left(\frac{\pi}{6}\right)\leq 3+\frac{1}{\Omega_{d-2}(\frac{\pi}{6})}\left(\frac{1}{2}+\frac{2\ln\ln(d-2)}{\ln(d-2)}+\frac{5}{\ln(d-2)}\right)(d-2)\ln(d-2),
\end{equation}
where $\Omega_{d-2}(\frac{\pi}{6})$ is the fraction of the surface of ${\mathbb S}^{d-2}$ covered by a closed spherical cap of angular radius $\frac{\pi}{6}$. Next, the estimate $\Omega_{d-2}(\frac{\pi}{6})>\frac{1}{2^{d-2}\sqrt{2\pi(d-1)}}$ (see for example, Lemma 2.1 in \cite{Na}) combined with (\ref{inequality1}) yields that

\begin{equation}\label{inequality2}
3+N_{{\mathbb S}^{d-2}}\left(\frac{\pi}{6}\right)\leq 3+2^{d-2}\sqrt{2\pi(d-1)}\left(\frac{1}{2}+\frac{2\ln\ln(d-2)}{\ln(d-2)}+\frac{5}{\ln(d-2)}\right)(d-2)\ln(d-2)<2^{d+1}d^{\frac{3}{2}}\ln d
\end{equation}
holds for all $d\geq 5$. Indeed, see Figure~\ref{graph of f(x)} for the graph of the function 
$$f(x):=\frac{3+2^{x-2}\sqrt{2\pi(x-1)}\left(\frac{1}{2}+\frac{2\ln\ln(x-2)}{\ln(x-2)}+\frac{5}{\ln(x-2)}\right)(x-2)\ln(x-2)}{2^{x+1}x^{\frac{3}{2}}\ln x} , x>3$$ 
which clearly implies the last inequality of (\ref{inequality2}). Finally, if  ${\rm Sp}_{\B^d}[\x_1,\dots ,\x_n]$ is a $2$-illuminable spiky (unit) ball with vertices $\x_1,\dots ,\x_n$ in $\Ee^d$, $d\geq 5$, then (\ref{inequality2}) combined with Part (iii) of Theorem~\ref{illumination-spiky balls}
finishes the proof of Corollary \ref{estimates-spiky balls}.

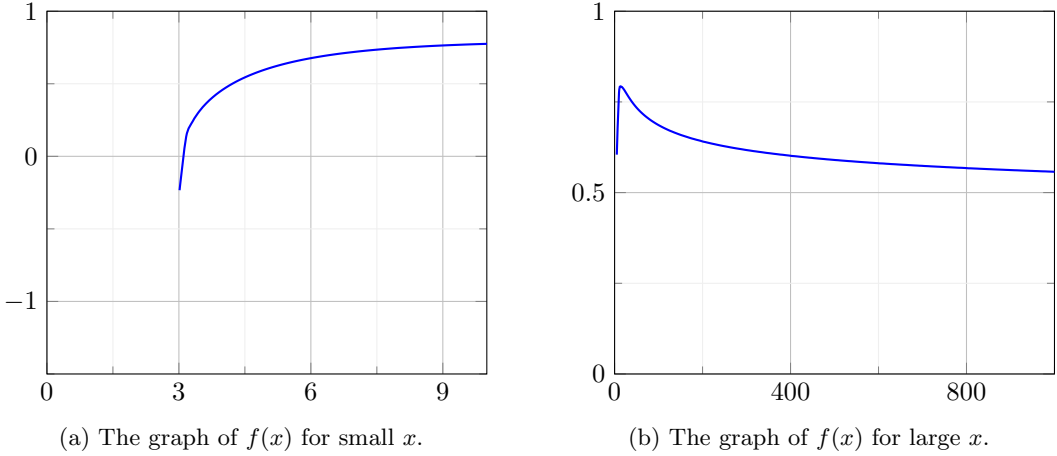
\begin{figure}
\centering
    \begin{subfigure}{.45\linewidth}
      \centering
      
      \begin{tikzpicture}
 
        \begin{axis}[
             xmin = 0, xmax = 10,
    ymin = -1.5, ymax = 1.0,
    xtick distance = 3,
    ytick distance = 1,
    grid = both,
    minor tick num = 1,
    major grid style = {lightgray},
    minor grid style = {lightgray!25},
    width = \textwidth,
     ]
        \addplot[
                domain = 0:30,
                samples = 200,
                smooth,
                thick,
                blue,
        ]{(3 + ( 2^(x-2) ) * ( (2*pi*(x-1))^0.5 ) * ( 0.5 + (2 * ln(ln(x-2)) / (ln(x-2))) + ( 5/(ln(x-2)) )) * (x-2) * ln(x-2)) / ( 2^(x+1) * x^1.5 * ln(x))};
    \end{axis}

\end{tikzpicture}
      \caption{The graph of $f(x)$ for small $x$.}
      
\end{subfigure}
    \begin{subfigure}{.45\textwidth}
      \centering

      \begin{tikzpicture}
 
        \begin{axis}[
             xmin = 0, xmax = 1000,
    ymin = 0, ymax = 1.0,
    xtick distance = 400,
    ytick distance = 0.5,
    grid = both,
    minor tick num = 1,
    major grid style = {lightgray},
    minor grid style = {lightgray!25},
    width = \textwidth,
   ]
    
        \addplot[
                domain = 0:1000,
                samples = 200,
                smooth,
                thick,
                blue,
        ]{(3 + ( 2^(x-2) ) * ( (2*pi*(x-1))^0.5 ) * ( 0.5 + (2 * ln(ln(x-2)) / (ln(x-2))) + ( 5/(ln(x-2)) )) * (x-2) * ln(x-2)) / ( 2^(x+1) * x^1.5 * ln(x))};
    \end{axis}
    
       ;
 
\end{tikzpicture}

      \caption{The graph of $f(x)$ for large $x$.}
     
    \end{subfigure}
  \caption{The graph of $f(x)$.}
  \label{graph of f(x)}
\end{figure}

\section{Proof of Remark \ref{construction1}}
Recall the following construction of Danzer \cite{Da}: there exist $10$ closed circular disks in $\Ee^2$ such that any two of them intersect and it is impossible to pierce them by $3$ needles. It follows in a straightforward way that there exists a family $\mathcal{C}$ of $10$ open circular disks in $\Ee^2$ (each being somewhat larger and concentric to a closed circular disk of the previous family) such that any two of them intersect and it is impossible to pierce them by $3$ needles. Now, Let $H$ be the plane tangent to ${\mathbb S}^{2}$ at the point say, $-\mathbf{s}$ with $\mathcal{C}$ lying in $H$. If we take the stereographic projection with center $\mathbf{s}$ that maps $H$ onto ${\mathbb S}^{2}\setminus\mathbf{s}$ and label the image of the family $\mathcal{C}$ by $\mathcal{C}'$, then $\mathcal{C}'$ is a family of $10$ open spherical caps in ${\mathbb S}^{2}$ such that any two of them intersect and it is impossible to pierce them by $3$ needles. By choosing $\mathcal{C}$ within a small neighbourhood $B_H(-\mathbf{s})$ of $-\mathbf{s}$ in $H$, we get that each member of $\mathcal{C}'$ is an open spherical cap of angular radius $<\frac{\pi}{2}$. Next, let us take the spiky unit ball ${\rm Sp} _{\B^3}[\x_1,\dots ,\x_{10}]$ with $\{C_{{\mathbb S}^{2}}(-\y_i, \frac{\pi}{2}-\alpha_i)\ |\ 1\leq i\leq 10\}=\mathcal{C}'$. Clearly, due to Part (a) of  Lemma~\ref{intersecting}, ${\rm Sp}_{\B^3}[\x_1,\dots ,\x_{10}]$ is $2$-illuminable. Finally, if we choose $B_H(-\mathbf{s})$ sufficiently small, such that the spherical caps of $\mathcal{C}'$ all lie in a hemisphere, then Part (b) of Lemma~\ref{intersecting} and Part (ii) of Theorem~\ref{illumination-spiky balls} yield that $I({\rm Sp}_{\B^3}[\x_1,\dots ,\x_{10}])=5$.

Next, we recall the following construction of Bourgain and Lindenstrauss \cite{BoLi}: there exists $d^*$ such that for any $d\geq d^*$ one possesses a finite point set $P$ of diameter $1$ in $\Ee^d$ whose any covering by unit diameter closed balls requires at least $1.0645^d$ balls. Hence, if we take the unit diameter closed balls centered at the points of $P$ in  $\Ee^d$, then any two balls intersect and it is impossible to pierce them by fewer than $\lfloor1.0645^d\rfloor$ needles. It follows in a straightforward way that for any $d\geq d^*$ there exists a family $\mathcal{C}_d$ of open balls centered at the points of $P$ in $\Ee^d$ (each being somewhat larger and concentric to a unit diameter closed ball of the previous family) such that any two of them intersect and it is impossible to pierce them by fewer than $\lfloor1.0645^d\rfloor$ needles. Now, Let $H$ be the hyperplane tangent to ${\mathbb S}^{d-1}$ at the point say, $-\mathbf{s}$ with $\mathcal{C}_{d-1}$ lying in $H$. If we take the stereographic projection with center $\mathbf{s}$ that maps $H$ onto ${\mathbb S}^{d-1}\setminus\mathbf{s}$ and label the image of the family $\mathcal{C}_{d-1}$ by $\mathcal{C}_{d-1}'$, then $\mathcal{C}_{d-1}'$ is a family of open spherical caps in ${\mathbb S}^{d-1}$ such that any two of them intersect and it is impossible to pierce them by fewer than $\lfloor1.0645^{d-1}\rfloor$ needles. By choosing $\mathcal{C}_{d-1}$ within a small neighbourhood $B_H(-\mathbf{s})$ of $-\mathbf{s}$ in $H$, we get that each member of $\mathcal{C}_{d-1}'$ is an open spherical cap of angular radius $<\frac{\pi}{2}$. Next, let us take the spiky unit ball ${\rm Sp}_{\B^d}[\x_1,\dots ,\x_{n}]$ with $\{C_{{\mathbb S}^{d-1}}(-\y_i, \frac{\pi}{2}-\alpha_i)\ |\ 1\leq i\leq n\}=\mathcal{C}_{d-1}'$. Clearly, due to Part (a) of  Lemma~\ref{intersecting}, ${\rm Sp}_{\B^d}[\x_1,\dots ,\x_{n}]$ is $2$-illuminable. Finally, if we choose $B_H(-\mathbf{s})$ sufficiently small, such that the spherical caps of $\mathcal{C}'_{d-1}$ all lie in a hemisphere, then Part (b) of Lemma~\ref{intersecting} yields that $I({\rm Sp}_{\B^d}[\x_1,\dots ,\x_{n}])\geq 1+\lfloor1.0645^{d-1}\rfloor$, where $d\geq d^*+1$.

\section{Proof of Theorem \ref{illumination-cap bodies}}

\begin{figure}[!h]
\centering
\begin{subfigure}{.5\textwidth}
  \centering
  
  \begin{tikzpicture} 
\node [above right, inner sep=0] (image) at (0,0) {\includegraphics[width=0.8\linewidth]{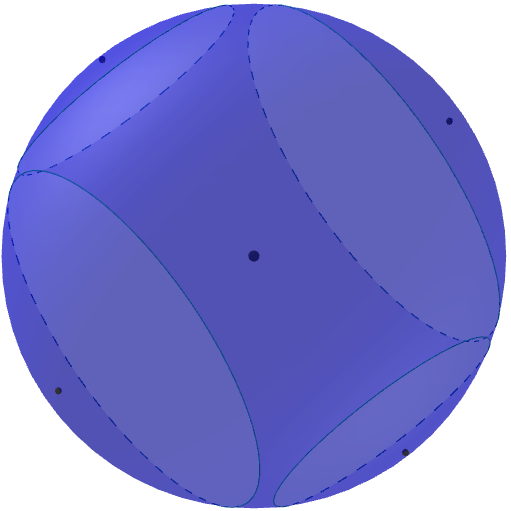}};    
    
\begin{scope}[
x={($0.1*(image.south east)$)},
y={($0.1*(image.north west)$)}];
        
    \node[left] at (2.5,9.5){\large $\y_1$};    
    \node[below left] at (1,2.5){\large $\y_2$};    
    \node[below right] at (8,1){\large $-\y_1$};
    \node[above right] at (9,7.5){\large $-\y_2$};
    \node[below right] at (5,5){\large $\oo$};
    
\end{scope}
\end{tikzpicture}
  \caption{Underlying packing of spherical caps of a cap body.}
  \label{fig:subcaps1}
\end{subfigure}%
\begin{subfigure}{.5\textwidth}
  \centering
    \begin{tikzpicture} 
\node [above right, inner sep=0] (image) at (0,0) {\includegraphics[width=0.8\linewidth]{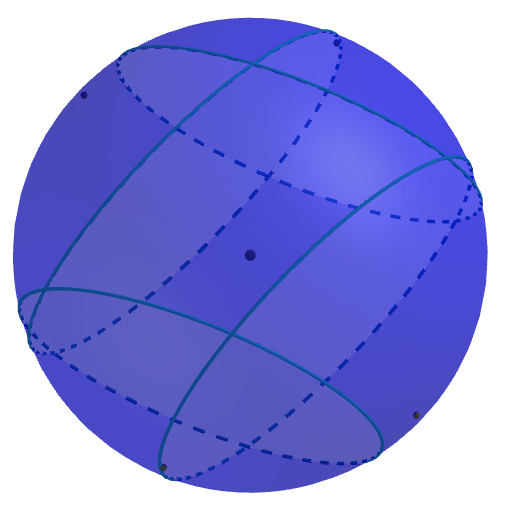}};    
    
\begin{scope}[
x={($0.1*(image.south east)$)},
y={($0.1*(image.north west)$)}];
       
    \node[left] at (1.75,8.5){\large $\y_1$};    
    \node[below left] at (3.5,0.75){\large $\y_2$};    
    \node[below right] at (8.25,2){\large $-\y_1$};
    \node[above right] at (6.5,9.25){\large $-\y_2$};
    \node[above right] at (5,5){\large $\oo$};
    
\end{scope}
\end{tikzpicture}
  \caption{Underlying arrangement of spherical caps of a non-convex spiky ball.}
  \label{fig:subcaps2}
\end{subfigure}
\caption{Underlying caps corresponding to the spiky balls in Figure \ref{fig:test}.}
\label{fig:caps}
\end{figure}
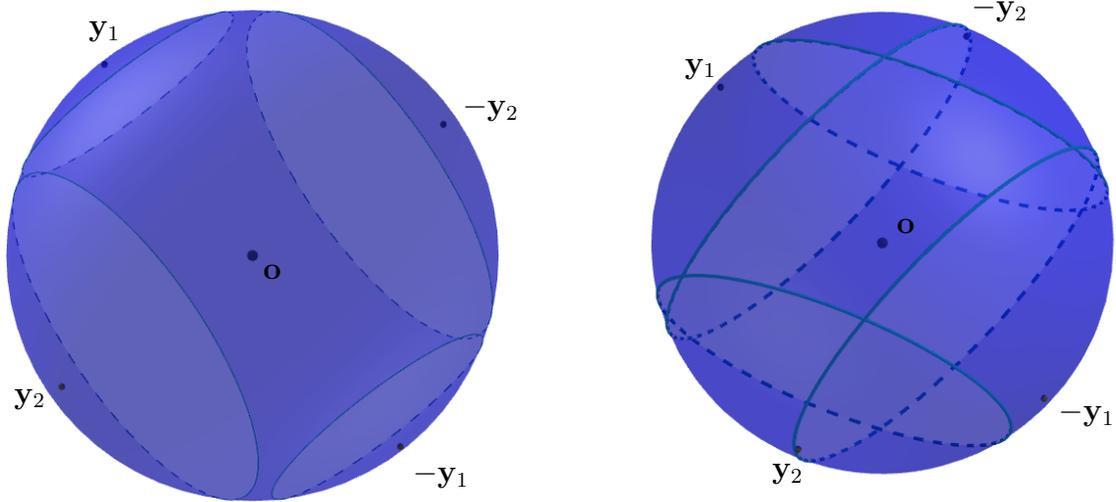

First, using Definition~\ref{underlying-open-caps} we prove

\begin{lemma}\label{o-symmetric-intersecting}
Let ${\rm Sp}_{\B^d}[\pm\x_1,\dots ,\pm\x_n]$ be an $\oo$-symmetric cap body with vertices $\pm\x_1,\dots ,\pm\x_n$ in $\Ee^d$, $d\geq 3$. Then
\begin{itemize}
\item[(a)] $C_{{\mathbb S}^{d-1}}[\pm\y_i, \frac{\pi}{2}-\alpha_i]\cap C_{{\mathbb S}^{d-1}}[\pm\y_j, \frac{\pi}{2}-\alpha_j]\neq\emptyset$ holds for all $1\leq i<j\leq n$ moreover,
\item[(b)] $\{\mathbf{v}_{k}: 1\leq k\leq m\}\subset {\mathbb S}^{d-1}$ with ${\rm pos}(\{\mathbf{v}_{k}: 1\leq k\leq m\})=\Ee^d$ illuminates ${\rm Sp}_{\B^d}[\pm\x_1,\dots ,\pm\x_n]$ if and only if $C_{{\mathbb S}^{d-1}}(\pm\y_i, \frac{\pi}{2}-\alpha_i)\cap\{\mathbf{v}_{k}: 1\leq k\leq m\}\neq\emptyset$ holds for all $1\leq i\leq n$.
\end{itemize}
\end{lemma}
\begin{proof}
Due to convexity and symmetry of ${\rm Sp}_{\B^d}[\pm\x_1,\dots ,\pm\x_n]$, the underlying spherical caps $C_{{\mathbb S}^{d-1}}[\pm\y_i, \alpha_i]$, $1\leq i\leq n$ form a packing in ${\mathbb S}^{d-1}$ (see the Figure \ref{fig:caps} for the examples of the spiky ball cap configurations). Now, let $1\leq i<j\leq n$. For Part (a) it is sufficient to show that 
\begin{equation}\label{intersection-a}
C_{{\mathbb S}^{d-1}}\left[-\y_i, \frac{\pi}{2}-\alpha_i\right]\cap C_{{\mathbb S}^{d-1}}\left[\y_j, \frac{\pi}{2}-\alpha_j\right]\neq\emptyset .
\end{equation} 
(Namely, the same argument and symmetry will imply that $C_{{\mathbb S}^{d-1}}[\pm\y_i, \frac{\pi}{2}-\alpha_i]\cap C_{{\mathbb S}^{d-1}}[\pm\y_j, \frac{\pi}{2}-\alpha_j]\neq\emptyset$.) Let $H_{ij}$ be a hyperplane passing through $\oo$ and separating  $C_{{\mathbb S}^{d-1}}[\y_i, \alpha_i]$ and $C_{{\mathbb S}^{d-1}}[\y_j, \alpha_j]$. Furthermore, let $\mathbf{n}_{ij}\in{\mathbb S}^{d-1}$ (resp., $-\mathbf{n}_{ij}\in{\mathbb S}^{d-1}$) be on the same side of $H_{ij}$ as $C_{{\mathbb S}^{d-1}}[\y_i, \alpha_i]$ (resp., $C_{{\mathbb S}^{d-1}}[\y_j, \alpha_j]$) such that $ \langle \pm\mathbf{n}_{ij},\z\rangle=0 $ for all $\z\in H_{ij}$. Clearly $-\mathbf{n}_{ij}\in C_{{\mathbb S}^{d-1}}\left[-\y_i, \frac{\pi}{2}-\alpha_i\right]$ moreover, $\mathbf{n}_{ij}\in C_{{\mathbb S}^{d-1}}\left[-\y_j, \frac{\pi}{2}-\alpha_j\right]$ implying $-\mathbf{n}_{ij}\in C_{{\mathbb S}^{d-1}}\left[\y_j, \frac{\pi}{2}-\alpha_j\right]$. Thus, (\ref{intersection-a}) follows, finishing the proof of Part (a). Finally, Part (b) follows from Part (b) of Lemma~\ref{intersecting} in a straightforward way.
\end{proof}

Second, based on Lemma~\ref{o-symmetric-intersecting}, in order to prove Theorem \ref{illumination-cap bodies} it is sufficient to show

\begin{theorem}\label{piercing o-symmetric spherical caps}
Let $\{C_{{\mathbb S}^{d-1}}[\pm\z_i, \beta_i]\ |\ 1\leq i\leq n\}\subset{\mathbb S}^{d-1}$ be an $\oo$-symmetric family of $2n$ closed spherical caps with $d\geq 3$ and $0<\beta_i<\frac{\pi}{2}$, $1\leq i\leq n$ such that
\begin{equation}\label{main-condition}
 C_{{\mathbb S}^{d-1}}[\pm\z_i, \beta_i]\cap C_{{\mathbb S}^{d-1}}[\pm\z_j, \beta_j]\neq\emptyset
 \end{equation}\label{11}
holds for all $1\leq i<j\leq n$. Then there exist $\mathbf{u}_1,\dots ,\mathbf{u}_N\in{\mathbb S}^{d-1}$ with $N=2+N_{{\mathbb S}^{d-2}}\left(\frac{\pi}{4}\right)$ and ${\rm pos}(\{\mathbf{u}_{k}: 1\leq k\leq N\})=\Ee^d$ such that
 $C_{{\mathbb S}^{d-1}}(\pm\z_i, \beta_i)\cap\{\mathbf{u}_{k}: 1\leq k\leq N\}\neq\emptyset$ holds for all $1\leq i\leq n$.
 \end{theorem}
 \begin{proof}
 Without loss of generality we may assume that the points $\{\pm\z_i\ |\ 1\leq i\leq n\}$ are pairwise distinct and
 \begin{equation}\label{11-a}
 0<\beta_1\leq\beta_2\leq\dots\leq\beta_n<\frac{\pi}{2}.
 \end{equation}
 Let $H$ be the hyperplane of $\Ee^d$ with normal vectors $\pm\z_1$ passing through $\oo$, and let ${\mathbb S}^{d-2}:=H\cap{\mathbb S}^{d-1}$.
 \begin{sublemma}\label{11-aa}
 ${\mathbb S}^{d-2}\cap C_{{\mathbb S}^{d-1}}[\pm\z_i, \beta_i]$ is a $(d-2)$-dimensional closed spherical cap of angular radius at least $\frac{\pi}{4}$ for all $2\leq i\leq n$.
 \end{sublemma}
 \begin{proof}
 Let $H^+$ be the closed halfspace of $\Ee^d$ bounded by $H$ that contains $\z_1$. Let $i$ be fixed with $2\leq i\leq n$. Without loss of generality we may assume that $\z_i\in H^+$ and our goal is to show that ${\mathbb S}^{d-2}\cap C_{{\mathbb S}^{d-1}}[\z_i, \beta_i]$ is a $(d-2)$-dimensional closed spherical cap of angular radius at least $\frac{\pi}{4}$. Let $\beta$ be the smallest positive real such that
\begin{equation}\label{subcondition}
\beta_1\leq\beta\leq\beta_i \ \text{and}\ C_{{\mathbb S}^{d-1}}[\z_i, \beta]\cap C_{{\mathbb S}^{d-1}}[-\z_1, \beta_1]\neq\emptyset \ (\text{and therefore also}\  C_{{\mathbb S}^{d-1}}[\z_i, \beta]\cap C_{{\mathbb S}^{d-1}}[\z_1, \beta_1]\neq\emptyset).
\end{equation}
Thus, either $C_{{\mathbb S}^{d-1}}[\z_i, \beta]$ is tangent to $C_{{\mathbb S}^{d-1}}[-\z_1, \beta_1]$ at some point of $\widehat{\z_i (-\z_1)}$ (Case 1) or $\beta_1=\beta$ (Case 2).

\bigskip
\noindent {\it Case 1:} Let $\mathbf{b}_i:=\widehat{\z_i (-\z_1)}\cap{\mathbb S}^{d-2}$ and $\mathbf{a}_i\in{\rm bd}\left(C_{{\mathbb S}^{d-1}}[\z_i, \beta]\right)\cap{\mathbb S}^{d-2}$, where ${\rm bd}(\cdot )$ denotes the boundary of the corresponding set in ${\mathbb S}^{d-1}$. If $\z_i\in H$, then $\z_i=\mathbf{b}_i$ and $\beta=l(\widehat{\mathbf{a}_i\mathbf{b}_i})$ and therefore, (\ref{subcondition}) yields $\frac{\pi}{4}=\frac{2\beta_1+2\beta}{4}\leq\beta$, finishing the proof of Sublemma~\ref{11-aa}. So, we are left with the case when $\mathbf{a}_i, \mathbf{b}_i$, and $\z_i$ are pairwise distinct points on ${\mathbb S}^{d-1}$ and the spherical triangle with vertices $\mathbf{a}_i, \mathbf{b}_i$, and $\z_i$ has a right angle at $\mathbf{b}_i$. Clearly, $l(\widehat{\mathbf{a}_i\z_i})=\beta$ and $l(\widehat{\mathbf{b}_i\z_i})=\beta_1+\beta-\frac{\pi}{2}$. Let $\gamma:=l(\widehat{\mathbf{a}_i\mathbf{b}_i})$. According to Napier's trigonometric rule for the side lengths of a spherical right triangle we have $\cos\beta=\cos\left(\beta_1+\beta-\frac{\pi}{2}\right)\cos\gamma$. As $\frac{\pi}{2}<\beta_1+\beta<\pi$ and $\beta_1\leq\beta<\frac{\pi}{2}$, it follows that
\begin{equation}\label{right-triangle}
\cos\gamma=\frac{\cos\beta}{\sin(\beta_1+\beta)}\leq\frac{\cos\beta}{\sin(2\beta)}=\frac{1}{2\sin\beta}<\frac{1}{2\sin\frac{\pi}{4}}=\frac{1}{\sqrt{2}}.
\end{equation}
Thus, $\gamma>\frac{\pi}{4}$, implying that the angular radius of ${\mathbb S}^{d-2}\cap C_{{\mathbb S}^{d-1}}[\z_i, \beta_i]$ is $>\frac{\pi}{4}$. This completes the proof of Sublemma~\ref{11-aa} in Case 1.

\bigskip
\noindent{\it Case 2:} Move $C_{{\mathbb S}^{d-1}}[\z_i, \beta]$ without changing its radius such that $\z_i$ moves along $\widehat{\z_i \z_1}$ and arrives at $\z_i^*\in \widehat{\z_i \z_1}$ with the property that $C_{{\mathbb S}^{d-1}}[\z_i^*, \beta]$ is tangent to $C_{{\mathbb S}^{d-1}}[-\z_1, \beta_1]$ at some point of $\widehat{\z_i^* (-\z_1)}$. Clearly, 
$${\mathbb S}^{d-2}\cap C_{{\mathbb S}^{d-1}}[\z_i^*, \beta]\subset{\mathbb S}^{d-2}\cap C_{{\mathbb S}^{d-1}}[\z_i, \beta] .$$
Thus, the proof of Case 1 applied to $C_{{\mathbb S}^{d-1}}[\z_i^*, \beta]$ finishes the proof of Sublemma~\ref{11-aa}.
 \end{proof}
 Now, let  $\mathbf{u}_1,\dots ,\mathbf{u}_{N_{{\mathbb S}^{d-2}}\left(\frac{\pi}{4}\right)}\in{\mathbb S}^{d-2}$ such that the $(d-2)$-dimensional closed spherical caps $C_{{\mathbb S}^{d-2}}[\mathbf{u}_j, \frac{\pi}{4}]$, $1\leq j\leq N_{{\mathbb S}^{d-2}}\left(\frac{\pi}{4}\right)$ cover ${\mathbb S}^{d-2}$. It follows via Sublemma~\ref{11-aa} that $\{\mathbf{u}_1,\dots ,\mathbf{u}_{N_{{\mathbb S}^{d-2}}\left(\frac{\pi}{4}\right)}\}\cap\left({\mathbb S}^{d-2}\cap C_{{\mathbb S}^{d-1}}[\pm\z_i, \beta_i]\right)\neq\emptyset$ holds for all $2\leq i\leq n$. If necessary one can reposition the points $\mathbf{u}_1,\dots ,\mathbf{u}_{N_{{\mathbb S}^{d-2}}\left(\frac{\pi}{4}\right)}$ by a properly chosen isometry in ${\mathbb S}^{d-2}$ such that the stronger condition  $\{\mathbf{u}_1,\dots ,\mathbf{u}_{N_{{\mathbb S}^{d-2}}\left(\frac{\pi}{4}\right)}\}\cap\left({\mathbb S}^{d-2}\cap C_{{\mathbb S}^{d-1}}(\pm\z_i, \beta_i)\right)\neq\emptyset$ holds as well for all $2\leq i\leq n$. Finally, adding the points $\pm\z_1$ to $\{\mathbf{u}_1,\dots ,\mathbf{u}_{N_{{\mathbb S}^{d-2}}\left(\frac{\pi}{4}\right)}\}$ completes the proof of Theorem~\ref{piercing o-symmetric spherical caps}.
  \end{proof}

\section{Proof of Corollary \ref{estimates-cap bodies}}

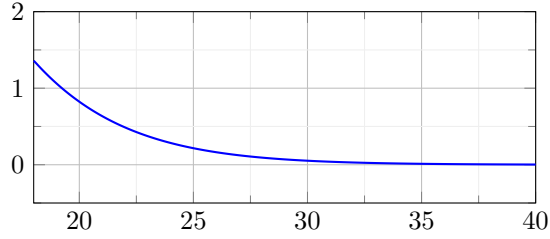
\begin{figure}[!h]
\centering
     \begin{tikzpicture}
 
        \begin{axis}[
     xmin = 18, xmax = 40,
     ymin = -0.5, ymax = 2,
     xtick distance = 5,
     ytick distance = 1,
     grid = both,
     minor tick num = 1,
     major grid style = {lightgray},
     minor grid style = {lightgray!25},
     width = 0.5\textwidth,
    height = 0.25\textwidth
    ]    
        \addplot[
                domain = 18:40,
                samples = 200,
                smooth,
                thick,
                blue,
        ]{ ( 2^(-x) ) * ( 2 + ( 2^((x-2)/2) ) * ((2*pi*(x-1))^0.5) * ( 0.5 + (2 * ln(ln(x-2)) / (ln(x-2))) + ( 5/(ln(x-2)) ) ) * (x-2) * ln(x-2) )};
        
    \end{axis}
\end{tikzpicture}
    \caption{The graph of $g(x)/h(x)$ for $x\geq 20$.}
    \label{g(x)/h(x)}
\end{figure}

Using $N_{{\mathbb S}^{1}}\left(\frac{\pi}{4}\right)=4$ and Theorem~\ref{illumination-cap bodies} one obtains in a straightforward way that any $3$-dimensional centrally symmetric cap body can be illuminated by $6$ directions in $\Ee^3$. Next, recall that $9\leq N_{{\mathbb S}^{2}}\left(\frac{\pi}{4}\right)\leq 10$ (\cite{TaGa}). This statement combined with Theorem~\ref{illumination-cap bodies} yields that any $4$-dimensional centrally symmetric cap body can be illuminated by $12$ directions in $\Ee^4$.

Finally,  recall that Theorem 1 of \cite{Du} implies in a straightforward way that 
\begin{equation}\label{inequality3}
2+N_{{\mathbb S}^{d-2}}\left(\frac{\pi}{4}\right)\leq 2+\frac{1}{\Omega_{d-2}(\frac{\pi}{4})}\left(\frac{1}{2}+\frac{2\ln\ln(d-2)}{\ln(d-2)}+\frac{5}{\ln(d-2)}\right)(d-2)\ln(d-2),
\end{equation}
where $\Omega_{d-2}(\frac{\pi}{4})$ is the fraction of the surface of ${\mathbb S}^{d-2}$ covered by a closed spherical cap of angular radius $\frac{\pi}{4}$. Hence, the estimate $\Omega_{d-2}(\frac{\pi}{4})>\frac{1}{2^{\frac{d-2}{2}}\sqrt{2\pi(d-1)}}$ (see for example, Lemma 2.1 in \cite{Na}) combined with (\ref{inequality3}) yields that

\begin{equation}\label{inequality4}
2+N_{{\mathbb S}^{d-2}}\left(\frac{\pi}{4}\right)\leq 2+2^{\frac{d-2}{2}}\sqrt{2\pi(d-1)}\left(\frac{1}{2}+\frac{2\ln\ln(d-2)}{\ln(d-2)}+\frac{5}{\ln(d-2)}\right)(d-2)\ln(d-2)<2^d
\end{equation}
holds for all $d\geq 20$. Indeed, Figure~\ref{g(x)/h(x)} shows that $g(x)/h(x)<1$ holds for all $x\geq 20$, where $$g(x):=2+2^{\frac{x-2}{2}}\sqrt{2\pi(x-1)}\left(\frac{1}{2}+\frac{2\ln\ln(x-2)}{\ln(x-2)}+\frac{5}{\ln(x-2)}\right)(x-2)\ln(x-2)$$ and $h(x):=2^x$. Thus, (\ref{inequality4}) combined with Theorem~\ref{illumination-cap bodies} finishes the proof of Corollary \ref{estimates-cap bodies}.

\section{Proof of Theorem \ref{unconditionally symmetric}}

             Theorem \ref{unconditionally symmetric} concerns illuminating the cap bodies ${\rm Sp}_{\B^d}[\pm\x_1,\dots ,\pm\x_n]$ that are symmetric about every coordinate hyperplane $H_j = \{\x \in \Ed | \langle \x, \ee_j \rangle =0 \}$, $1\leq j\leq d$ in $\Ed$. According to Lemma \ref{o-symmetric-intersecting}, we only need to show that the open spherical caps $\copp{\y_i}{\pi/2-\alpha_i}, 1\leq i \leq n$ can be pierced by $4d$ points in $\Sd$  such that the positive hull of these $4d$ unit vectors is $\Ed$.

             We start by trying to use the $2d$ points $\{\pm \ee_j \mid 1 \leq j \leq d \}$. If all the above mentioned open spherical caps are pierced by these $2d$ points, then the cap body in question can be illuminated by $2d$ directions and we are done. So, suppose there is a vertex $\x_i$ such that the cap $\copp{\y_i}{\pi/2-\alpha_i}$ isn't pierced by any of the $2d$ points $\{\pm \ee_j \mid 1 \leq j \leq d \}$. Suppose then that $k\geq0$ of the points $\pm \ee_j, 1 \leq j \leq d$ lie on the boundary of this cap, and the rest of these points are not in the cap's closure $\capp{\y_i}{\pi/2-\alpha_i}$. This leads us to

             \begin{definition}\label{def:kspan}
                An open spherical cap $\copp{\y_i}{\pi/2-\alpha_i}$ is a {\rm $k$-spanning cap} if exactly $k\geq 0$ points of the set $\left\{\pm \ee_j \mid 1 \leq j \leq d\right\}$ lie on the boundary of $\copp{\y_i}{\pi/2-\alpha_i}$, and the other $2d-k$ points from the set $\left\{\pm \ee_j \mid 1 \leq j \leq d\right\}$ do not belong to $\capp{\y_i}{\pi/2-\alpha_i}$. The images of a $k$-spanning cap under arbitrary composition of finitely many reflections about the coordinate hyperplanes of  $\Ed$ is called a {\rm $k$-spanning family of caps}.
             \end{definition}

             To properly study these $k$-spanning families, we need the following fact: the underlying spherical caps $C_{{\mathbb S}^{d-1}}[\pm\y_i, \alpha_i]$, $1\leq i\leq n$ of ${\rm Sp}_{\B^d}[\pm\x_1,\dots ,\pm\x_n]$ form a packing in ${\mathbb S}^{d-1}$. This means $ \alpha_i + \alpha_j \leq l(\widehat{\y_i, \y_j}) $ for any $i\neq j \in \{1, \dots, n\}$. For the piercing caps $\copp{\y_i}{\pi/2-\alpha_i}$ and $\copp{\y_j}{\pi/2-\alpha_j}$  we can rewrite this condition as
	
            \begin{equation}\label{paquiv}
                	\left(\pi/2 - \alpha_i\right)+\left(\pi/2 - \alpha_j\right) \geq \pi - l(\widehat{\y_i, \y_j}).            
            \end{equation}

            \begin{lemma} \label{lem:kspan}
                The open spherical cap $\copp{\y}{\varphi}$ with $0<\varphi<\frac{\pi}{2}$ belongs to some $k$-spanning family if and only if the coordinates of $\y$ form a permutation of the sequence $\underbrace{\pm 1/\sqrt{k}, \dots, \pm 1/\sqrt{k}}_k, \underbrace{0, \dots, 0}_{d-k}$ and $\varphi$ is equal to $\arccos (1/k)$, where $2\leq k\leq d$.
            \end{lemma}
            
            \begin{proof}
            
             The statement is trivial in one direction. Namely, it is clear that the open spherical cap with the centre and radius as described is not pierced by any vectors from $\{\pm \ee_j | 1 \leq j \leq d\}$. In particular, it is a $k$-spanning cap for $2\leq k\leq d$. It is also clear, that its images under arbitrary composition of finitely many reflections about the coordinate hyperplanes of  $\Ed$ are $k$-spanning caps as well, forming a $k$-spanning family. 
            
                        So, we are left to prove the non-trivial direction. Since any $k$-spanning family is unconditionally symmetric, we may assume that the cap $\copp{\y}{\varphi}$ with centre $
            \y = (x_1, \dots, x_d)$ is such that the $x_j$'s are non-negative. Without loss of generality, suppose $\ee_1, \dots, \ee_k \in \bd~ \capp{\y}{\varphi}$ and $\ee_{k+1}, \dots, \ee_d$ are not in $\capp{\y}{\varphi}$, where $0<\varphi < \pi/2$. We can rewrite this condition as follows: 

            \begin{equation}\label{eq:coord}
                \begin{cases}
                            x_j = \cos \varphi, \mbox{if } j \leq k \\
                            x_j < \cos \varphi, \mbox{if } j > k
                \end{cases}
            \end{equation}

                    To finish the proof we only need to show that $x_{k+1} = x_{k+2} =, \dots= x_d=0$. Suppose $x_{k+1} > 0$. Then let the cap $\copp{\y'}{\varphi}$ be a reflection of $\copp{\y}{\varphi}$ about the coordinate hyperplane $H_{k+1}$, hence $\y'=(x_1, \dots, x_k, -x_{k+1}, x_{k+2}, \dots, x_d)$. Using the inequality (\ref{paquiv}) for the caps $\copp{\y'}{\varphi}$ and the $\copp{\y}{\varphi}$, we get the following:

            \begin{align}
            \varphi+\varphi &\geq \pi - l(\widehat{\y, \y'})  \\
            \cos(2\varphi) &\leq -\cos(l(\widehat{\y, \y'}))\\
                \cos 2\varphi &\leq -(x_1^2 + \dots + x_k^2 -x_{k+1}^2 + x_{k+2}^2 + \dots + x_d^2) \\
                2\cos^2 \varphi-1 &\leq -1 +2x_{k+1}^2\\
                \cos \varphi &\leq x_{k+1}
            \end{align}

         That clearly contradicts the second part of (\ref{eq:coord}) and so, $x_{k+1}=0$, and the same goes for $x_{k+2},\dots,x_d$. Finally, by the first part of (\ref{eq:coord}) one obtains that $\varphi = \arccos(1/\sqrt{k})$. It follows via $0<\varphi < \pi/2$ that $2\leq k \leq d$, finishing the proof of Lemma~\ref{lem:kspan}.
            \end{proof}

\begin{claim}\label{clm:escape} The open spherical cap $C_{{\mathbb S}^{d-1}}(\x, \alpha)$ is pierced by $\uu \in \Sd$ or $-\uu \in \Sd$ if and only if $\left|\frac{\langle \x,\uu \rangle}{ \cos \alpha} \right| >1 $. \end{claim}
\begin{proof}
	Clearly, $\uu$ pierces $C_{{\mathbb S}^{d-1}}(\x, \alpha)$ if and only if $l(\widehat{\uu,\x}) < \alpha$, i.e., $\cos \left(l(\widehat{\uu,\x})\right) > \cos \alpha$. 
	As $\alpha$ ranges from 0 to $\frac{\pi}{2} $ it is equivalent to $\frac{\langle \uu, \x \rangle}{\cos \alpha} >1$. Similarly, $-\uu$ piercing $C_{{\mathbb S}^{d-1}}(\x, \alpha)$ is equivalent to $\frac{\langle \uu, \x \rangle}{\cos \alpha}<-1$. Bringing these statements together gives us the claim. 
\end{proof}

\begin{figure}
    \centering
    
\begin{tikzpicture}
\node [above right, inner sep=0] (image) at (0,0) {\includegraphics[scale=0.4]{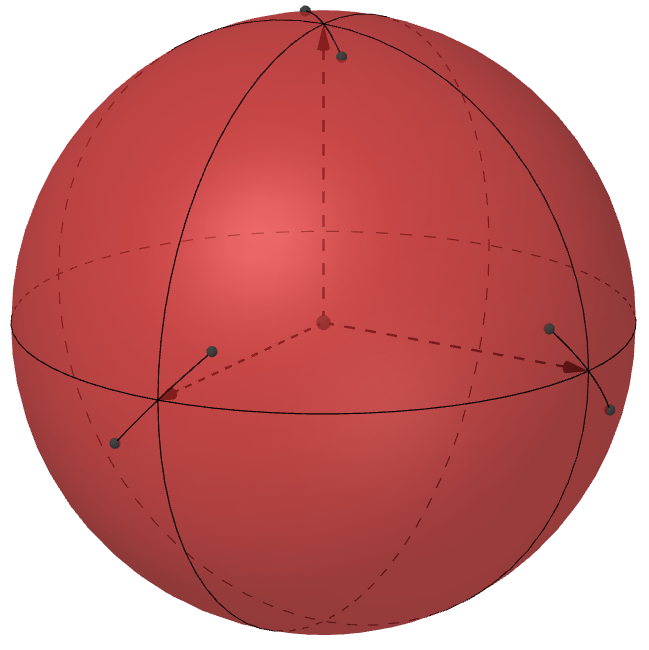}};

   \begin{scope}[
x={($0.1*(image.south east)$)},
y={($0.1*(image.north west)$)}];

   \node[above right] at (5,5){$\mathbf{o}$};
    
    \node[left] at (3.3,4){$\ee_1$};

    \node[above right] at (8,4){$\ee_2$};
    
    \node[below left] at (9.7,4.3){$\varphi$};
    
    \node[left] at (5,9){$\ee_3$};
    
    \node[above left] at (3.05,4.15){$\varphi$};
    
    \node[above right] at (3.35,4.4){$\uu_1$};
    
    \node[above left] at (8.3,4.7){$\uu_2$};
    
    \node[below right] at (5.2,8.9){$\uu_3$};
    
    \node[below right] at (1.8,3.1){$\vv_1$};
    
    \node[below left] at (9.5,3.6){$\vv_2$};
    
    \node[above left] at (4.7,9.7){$\vv_3$};
    
\end{scope}

\end{tikzpicture}

    \caption{Example of the points $\uu_j, \vv_j$ on $\Ss^2$.}
    \label{fig:perturb}
\end{figure}

Each non-$k$-spanning cap is pierced by some point from  $\{\pm \ee_j \mid 1 \leq j \leq d \}$. If we take our new piercing points close enough to $\{\pm \ee_j \mid 1 \leq j \leq d \}$, we still pierce all the non-$k$-spanning caps. Thus, we only need to construct a set of $4d$ points on $\Sd$ such that there is a point in a sufficiently small neighbourhood of every point from $\{\pm \ee_j \mid 1 \leq j \leq d \}$ moreover, every $k$-spanning cap is pierced.

\begin{lemma}\label{lem:4dpierce} Let $\varphi$ be an angle in $(0, \pi/2)$, and let the points $\uu_j,\vv_j \in \Sd, 1 \leq j \leq d$ be defined in the following way:

$$\uu_j = \left(\underbrace{\frac{\sin \varphi}{\sqrt{d-1}}, \frac{\sin \varphi}{\sqrt{d-1}}, \dots, \frac{\sin \varphi}{\sqrt{d-1}}, \cos \varphi}_{j}, \frac{\sin \varphi}{\sqrt{d-1}}, \dots, \frac{\sin \varphi}{\sqrt{d-1}}\right), $$

$$\vv_j = \left(\underbrace{-\frac{\sin \varphi}{\sqrt{d-1}}, -\frac{\sin \varphi}{\sqrt{d-1}}, \dots, -\frac{\sin \varphi}{\sqrt{d-1}}, \cos \varphi}_{j}, -\frac{\sin \varphi}{\sqrt{d-1}}, \dots, -\frac{\sin \varphi}{\sqrt{d-1}}\right).$$

If $\varphi$ is sufficiently small, then the $4d$ vectors $\left\{\pm \uu_j, \pm \vv_j \mid 1 \leq j \leq d \right\}$ pierce any $k$-spanning cap.

\end{lemma}

\begin{proof}

 Essentially, as seen in the Figure \ref{fig:perturb}, we obtain $\uu_j$ by rotating $\ee_j$ with an angle $\varphi$ towards the point $\left(1/\sqrt{k}, \dots, 1/\sqrt{k} \right)$, and $\vv_j$ we get by rotating $\ee_j$  away from the same point. 

Let $C_{{\mathbb S}^{d-1}}(\y, \alpha)$ be an open spherical cap of a $k$-spanning family. Lemma~\ref{lem:kspan} implies that $\alpha = \arccos (1/\sqrt{k})$ and $\y = \frac{1}{\sqrt{k}}(s_1, \dots, s_d)$ such that $s_j \in \{0,\pm 1\}$ and $\sum_{j=1}^{d}s_j^2=k$, where $2\leq k\leq d$. We will need the parameter $s = \sum_{j=1}^d s_j$ as well. Next, we pick some $1 \leq j \leq d$ such that $s_j \neq 0$. According to Claim \ref{clm:escape}, $C_{{\mathbb S}^{d-1}}(\y, \alpha)$ is pierced by $\uu_j$ or $-\uu_j$ (resp., $\vv_j$ or $-\vv_j$) if and only if $\left| \langle \uu_j, \sqrt{k}\y \rangle \right| >1$ (resp., $\left| \langle \vv_j, \sqrt{k}\y \rangle \right| >1$). Now, observe that 

\begin{equation}\label{dot-product}
\langle \uu_j, \sqrt{k}\y \rangle = s_j\cos \varphi + (s-s_j) \frac{\sin \varphi}{\sqrt{d-1}}\ {\rm and}\ \langle \vv_j, \sqrt{k}\y \rangle = s_j\cos \varphi - (s-s_j) \frac{\sin \varphi}{\sqrt{d-1}}.
\end{equation}

If $s_j(s-s_j) >0$, then (\ref{dot-product}) implies that for any sufficiently small $\varphi$ one has $\left|\langle \uu_j, \sqrt{k}\y \rangle\right|>1$. Similarly, if $s_j(s-s_j)<0$, then by (\ref{dot-product}) $\left|\langle \vv_j, \sqrt{k}\y \rangle\right|>1$ holds for any sufficiently small $\varphi$. So, we are left with the case when $s_j(s-s_j) =0$. Since $s_j\neq0$, that yields $s_j=s$. Thus, $s=\pm1$ and so, we just pick a different $j$ so that $s\neq s_j$ and repeat the above process. Indeed, we can do that since $s=\pm1$, and that means we must have both 1's and -1's in the sequence $s_1,\dots,s_d$. Otherwise, the sign of all the non-zero $s_j$'s would be the same, and that would result in $s=\pm k$, a contradiction because $k\geq2$. This completes the proof of Lemma~\ref{lem:4dpierce}.\end{proof}

Clearly, the positive hull of the vectors $\left\{\pm \uu_j, \pm \vv_j \mid 1 \leq j \leq d \right\}$ is $\Ed$. Moreover, if $\varphi$ is sufficiently small, then any cap $\copp{\y_i}{\pi/2-\alpha_i}$ that isn't a $k$-spanning cap for some $1\leq i\leq n$ and $2\leq k\leq d$, is pierced by a point from $\left\{\pm \uu_j, \pm \vv_j \mid 1 \leq j \leq d \right\}$. Finally, if $\copp{\y_i}{\pi/2-\alpha_i}$ is a $k$-spanning cap for some $1\leq i\leq n$ and $2\leq k\leq d$, then Lemma~\ref{lem:4dpierce} implies that it is pierced by the $4d$ vectors $\left\{\pm \uu_j, \pm \vv_j \mid 1 \leq j \leq d \right\}$. That concludes the proof of Theorem \ref{unconditionally symmetric}.

\bigskip

\noindent K\'aroly Bezdek \\
\small{Department of Mathematics and Statistics, University of Calgary, Canada}\\
\small{Department of Mathematics, University of Pannonia, Veszpr\'em, Hungary}\\
\small{E-mail: \texttt{bezdek@math.ucalgary.ca}}

\bigskip

\noindent Ilya Ivanov \\
\small{Department of Mathematics and Statistics, University of Calgary, Canada}\\
\small{E-mail: \texttt{ilya.ivanov1@ucalgary.ca}}

\bigskip

\noindent Cameron Strachan \\
\small{Department of Mathematics and Statistics, University of Calgary, Canada}\\
\small{E-mail: \texttt{braden.strachan@ucalgary.ca}}

\end{document}